\documentclass{amsart}
\usepackage{amssymb}
\usepackage{amsthm}

\newtheorem{prop}{Proposition}[section]

\newtheorem{conj}{Conjecture}[section]
\theoremstyle{definition} \newtheorem{defin}{Definition}[section]
\newtheorem{ex}{Example}[section] \theoremstyle{remark}
\newtheorem{rem}{Remark}[section] \newcommand{\pn}{\par\noindent}
\newcommand{\pmn}{\par\medskip\noindent}
\begin{document}
\title{Mathematical and physical billiard in pyramids}
\author{Yury Kochetkov and Lev Pyatko}
\date{}
\begin{abstract} In this experimental work we study billiard
trajectories in triangular pyramids and try to establish
conditions that guarantee the existence (or absence) of 4-cycles
(there can be not more, than three of them). We formulate
conjectures and prove some statements. For example, if a pyramid
has two orthogonal faces, then it has not more than two 4-cycles.
Also we study 4-cycles of the "physical" billiard in pyramids,
i.e. in the presence of gravity. Here we present our observations
for a generic case. \end{abstract}

\email{yukochetkov@hse.ru, lapyatko@edu.hse.ru} \maketitle

\section{Introduction}
\pn The two-dimensional mathematical billiard is a vast and well
known part of geometry (see \cite{GZ} or \cite{Ta}). However, the
3-dimensional billiard is much less known and the physical
billiard (in the presence of gravity) is, probably, a novel area
of investigation. We are interested in the simplest problem here:
the existence of 4-cycles in triangular pyramids. Cycles in
pyramids were studied in \cite{TZ} and \cite{Be}. In \cite{TZ}
periodic trajectories in right-angled tetrahedra were constructed,
but there 4-cycles do not exist. In \cite{Be} only pyramids in a
neighborhood of the regular tetrahedron were considered. \pmn Our
work is organized in the following way. In the second section we
explain how one can find a 4-cycle with the given order of
reflections. In the third section we introduce the "map of
cycles", i.e. for the given base of a pyramid and a variable
altitude we study 4-cycles with the given order of reflections. In
the forth section we prove statements about existence and behavior
of 4-cycles in some special cases. In the fifth section we present
results of our investigations in the physical case, i.e. in the
presence of gravity.

\section{Construction of 4-cycles}
\pn Usually we will work with pyramid $ABCD$ with base $ABC$ in
the $xy$ plane, the apex $D$ is in the upper half-space, $A$ ---
at the origin, $B$ --- at the positive $x$-axis, $C$ --- in the
upper half plane, $DO$ will be the altitude and $h=|DO|$. Each
4-cycle is determined by an order of reflections. There are three
possible trajectories, because each trajectory can be passed in
both directions:
\[\begin{picture}(370,90) \put(10,20){\line(1,0){80}}
\put(10,20){\line(5,6){50}} \qbezier[60](10,20)(55,35)(100,50)
\put(90,20){\line(1,3){10}} \put(90,20){\line(-1,2){30}}
\put(60,80){\line(4,-3){40}} \put(4,17){\tiny A} \put(93,17){\tiny
B} \put(64,79){\tiny D} \put(103,48){\tiny C}
\put(60,25){\circle*{2}} \put(40,40){\circle*{2}}
\put(55,50){\circle*{2}} \put(85,50){\circle*{2}}
\put(60,25){\vector(-4,3){19}} \put(40,40){\vector(3,2){14}}
\put(55,50){\vector(1,0){29}} \put(85,50){\vector(-1,-1){24}}
\put(10,11){\tiny ABC $\to$ ABD $\to$ ACD $\to$} \put(20,1){\tiny
$\to$ BCD $\to$ ABC}

\put(140,20){\line(1,0){80}} \put(140,20){\line(5,6){50}}
\qbezier[60](140,20)(185,35)(230,50) \put(220,20){\line(1,3){10}}
\put(220,20){\line(-1,2){30}} \put(190,80){\line(4,-3){40}}
\put(134,17){\tiny A} \put(223,17){\tiny B} \put(194,79){\tiny D}
\put(233,48){\tiny C} \put(200,35){\circle*{2}}
\put(180,45){\circle*{2}} \put(200,35){\vector(-2,1){19}}
\put(170,35){\circle*{2}} \put(180,45){\vector(-1,-1){10}}
\put(215,50){\circle*{2}} \put(170,35){\vector(3,1){44}}
\put(215,50){\vector(-1,-1){14}} \put(140,11){\tiny ABC $\to$ ACD
$\to$ ABD $\to$} \put(150,1){\tiny $\to$ BCD $\to$ ABC}

\put(270,20){\line(1,0){80}} \put(270,20){\line(5,6){50}}
\qbezier[60](270,20)(315,35)(360,50) \put(350,20){\line(1,3){10}}
\put(350,20){\line(-1,2){30}} \put(320,80){\line(4,-3){40}}
\put(264,17){\tiny A} \put(353,17){\tiny B} \put(324,79){\tiny D}
\put(363,48){\tiny C} \put(315,30){\circle*{2}}
\put(325,45){\circle*{2}} \put(315,30){\vector(2,3){9}}
\put(345,50){\circle*{2}} \put(325,45){\vector(4,1){19}}
\put(305,50){\circle*{2}} \put(345,50){\vector(-1,0){39}}
\put(305,50){\vector(1,-2){10}} \put(270,11){\tiny ABC $\to$ ABD
$\to$ BCD $\to$} \put(280,1){\tiny $\to$ ACD $\to$ ABC}
\end{picture}\]  \begin{center} Figure 1 \end{center}
\pmn Let $p_0$ be the starting point of our trajectory (usually it
will be a point in the base $ABC$), $\bar v$ be the starting
vector and $F_1F_2F_3F_4$ be the order of reflections. Each
reflection defines an orthogonal operator $R_i$ with the matrix
$M_i$. The composition $R_1\circ R_2\circ R_3\circ R_4$ is a
rotation with the matrix $M=M_4M_3M_2M_1$. Thus, $\bar v$ is the
eigenvector of $M$ with the eigenvalue 1.
\begin{rem} $R_1\circ R_2$ is a rotation around an edge and $R_3\circ R_4$
also is a rotation around another edge. These two rotations are
rotations around \emph{skew lines}, hence, the composition
$R_1\circ R_2\circ R_3\circ R_4$ cannot be the identity. \end{rem}
\begin{ex} Let $A=(0,0,0)$, $B=(4,0,0)$, $C=(2,4,0)$ and $D=(2,3,3)$, $p_0\in
ABC$ and  the order of reflections be $ABD\to ACD\to BCD\to ABC$.
Let $M_1,M_2,M_3,M_4$ be matrices of reflections with respect to
planes $ABD,ACD,BCD,ABC$, respectively. Here {\scriptsize
$$M_4M_3M_2M_1\!=\!\left(\!\!\!\!\begin{array}{ccr} 1&0&0\\
0&1&0\\ 0&0&-1\end{array}\!\!\!\!\right)\!\!
\left(\!\!\!\!\begin{array}{rrr}-13/23&-18/23&-6/23\\ -18/23&14/23&-3/23\\
-6/23&-3/23&22/23\end{array}\!\!\!\!\right)\!\!
\left(\!\!\!\!\begin{array}{rrr}
-13/23&18/23&6/23\\ 18/23&14/23&-3/23\\
6/23&-3/23&22/23\end{array}\!\!\!\!\right)\!\!
\left(\!\!\!\!\begin{array}{ccc} 1&0&0\\
0&0&1\\0&1&0\end{array}\!\!\!\!\right)\!=\!
\frac{1}{529}\!\!\left(\!\!\!\!\begin{array}{rrr} -191&-156&-468\\
468&-216&-119\\ -156&-457&216\end{array}\!\!\!\!\right)\!=\!M$$}
The vector $\bar v=(-13,-12,24)$ is the the eigenvector with
eigenvalue $1$. \pmn Now we must find the starting point. We
construct the pyramid $ABC_1D$, where $C$ and $C_1$ are symmetric
with respect to the plane $ABD$. Then we construct the pyramid
$AB_1C_1D$, where $B$ and $B_1$ are symmetric with respect to the
plane $AC_1D$. And then we construct the pyramid $A_1B_1C_1D$,
where $A$ and $A_1$ are symmetric with respect to the plane
$B_1C_1D$. Here $C_1=(2,0,4)$,
$B_1=\left(-\frac{52}{23},\frac{24}{23},\frac{72}{23}\right)$,
$A_1=\left(-\frac{432}{529},\frac{1176}{529},\frac{3528}{529}\right)$.
We replace a polygonal trajectory of our 4-cycle by a line that
connects a point $F\in\Delta ABC$ with barycentric coordinates
$(x,y,z)$ and a point $F_1\in \Delta A_1B_1C_1$ with the same
barycentric coordinates. If $F$ is the starting point of the
4-cycle, then the vector $\overline{FF}_1$ is collinear to the
vector $\bar v$. This condition defines $x$, $y$ and $z$:
{\scriptsize
$$x=\frac{115}{1778}\,,\,\, y=\frac{583}{1778}\,,\,\,
z=\frac{540}{889} \Rightarrow F=\left(\frac{2246}{889},
\frac{2160}{889}, 0 \right)\,.$$} It remains to check that the
line $FF_1$ intersects planes $B_1C_1D$, $AC_1D$ and $ABD$
\emph{inside} triangles $\Delta B_1C_1D$, $\Delta AC_1D$ and
$\Delta ABD$, respectively. \end{ex}

\begin{prop} For a given order of reflections we have either one
4-cycle, or none. \end{prop}
\begin{proof} After constructing pyramids $ABC_1D$, $AB_1C_1D$ and
$A_1B_1C_1D$, we solve a linear system to find a starting point
$F$. But a linear system has either a unique solution, or
infinitely many (a line), or none. We must demonstrate that
infinite case is impossible. \pmn Let us assume that points of a
line $\ell\subset ABC$ are solutions of our system. Points with
the same barycentric coordinates constitute a line $\ell_1\subset
A_1B_1C_1$. Let points $P\in \ell$ and $P_1\in \ell_1$ have the
same coordinates and points $Q\in \ell$ and $Q_1\in \ell_1$ also.
Then lines $PP_1$ and $QQ_1$ are parallel. As $|PQ|=|P_1Q_1|$,
then $PQQ_1P_1$ is a parallelogram, i.e. $\ell\parallel \ell_1$.
\pmn Let $R$ be the composition of reflections with respect to
planes $ABD,AC_1D,B_1C_1D$. Then R is a rotation $\rho$ around
some line $L$ and a reflection $\pi$ with respect to the plane
$\Pi$, orthogonal to $L$. As $R$ maps $\ell$ into $\ell_1$, then
$\ell$ is parallel to $L$. But the reflection $\pi$ changes the
barycentric order of points in $\ell$ into opposite.
\end{proof}

\section{Computations and conjectures}
\pn Let $ABCD$ be a pyramid, where $\Delta ABC$ is a fixed acute
triangle in the plane $xy$. The point $O$ --- the base of the
altitude $DO$ is fixed, but the height $h=|DO|$ is variable. Let
$ABC\to ACD\to ABD\to BCD \to ABC$ be the order of reflections.
Then a 4-cycle with this order of reflections either
\begin{enumerate} \item exists for all $h>a\geqslant 0$;
\item or exists for $0\leqslant a<h<b<\infty$; \item or does not
exist for all $h>0$.
\end{enumerate}
\begin{defin} Points $O$ for which the first case is realized
constitute the $\alpha$-set. Points $O$ for which the second case
is realized constitute the $\beta$-set. Points $O$ for which the
third case is realized constitute the $\gamma$-set. The
arrangement of these three sets in the plane will be called the
\emph{the map of cycles}.
\end{defin}
\begin{rem} $\alpha$-set for the cycle $ABC\to ACD\to ABD\to
BCD\to ABC$ and $\alpha$-set for the cycle $ABC\to ABD\to ACD\to
BCD\to ABC$ are, of course, different sets.
\end{rem}
\pn How one can describe these sets? We can give only a partial
answer. Let $ABC$ be a fixed acute triangle: $A$ at origin, $B$ at
the positive $x$-axis, $C$ at the upper half-plane. We will
describe the map of cycles for the order $ABC\to ACD\to ABD\to
BCD\to ABC$. The construction of the map is performed in the
following steps (Figure 2): a) we rotate $\Delta ABC$ on $\pi$
around the center of $AB$ and obtain the triangle $ABC'$; b) we
draw altitudes $AG,AG',BF,BF'$; c) we draw the lines $CC'$, $FF'$
and $GG'$ (the last two lines are parallel).
\[\begin{picture}(400, 90)
    \put(10, 50){\line(1, 0){60}}
    \put(10, 50){\line(1, 2){20}}
    \put(30, 90){\line(1, -1){40}}
    \put(4, 48){\tiny A}
    \put(71, 48){\tiny B}
    \put(32, 89){\tiny C}

    \put(85, 48){$\Rightarrow$}

    \put(110, 50){\line(1, 0){60}}
    \put(110, 50){\line(1, 2){20}}
    \put(130, 90){\line(1, -1){40}}
    \put(104, 48){\tiny A}
    \put(171, 48){\tiny B}
    \put(132, 89){\tiny C}
    \put(110, 50){\line(1, -1){40}}
    \put(150, 10){\line(1, 2){20}}
    \put(152, 7){\tiny C$'$}

    \put(185, 48){$\Rightarrow$}

    \put(210, 50){\line(1, 0){60}}
    \put(210, 50){\line(1, 2){20}}
    \put(230, 90){\line(1, -1){40}}
    \put(204, 48){\tiny A}
    \put(271, 48){\tiny B}
    \put(232, 89){\tiny C}
    \put(210, 50){\line(1, -1){40}}
    \put(250, 10){\line(1, 2){20}}
    \put(252, 7){\tiny C$'$}
    \qbezier[40](210, 50)(225, 65)(240, 80)
    \qbezier[40](270, 50)(255, 35)(240, 20)
    \qbezier[40](210, 50)(235, 37.5)(258, 26)
    \qbezier[40](270, 50)(246, 62)(222, 74)
    \put(217, 74){\tiny F}
    \put(231, 17){\tiny F$'$}
    \put(240, 20){\circle{1}}
    \put(222, 74){\circle{1}}
    \put(242, 79){\tiny G}
    \put(260, 24){\tiny G$'$}
    \put(258, 26){\circle{1}}
    \put(240, 80){\circle{1}}

    \put(285, 48){$\Rightarrow$}

    \put(310, 50){\line(1, 0){60}}
    \put(310, 50){\line(1, 2){20}}
    \put(330, 90){\line(1, -1){40}}
    \put(304, 48){\tiny A}
    \put(371, 48){\tiny B}
    \put(330, 91){\tiny C}
    \put(310, 50){\line(1, -1){40}}
    \put(350, 10){\line(1, 2){20}}
    \put(352, 7){\tiny C$'$}
    \qbezier[40](310, 50)(325, 65)(340, 80)
    \qbezier[40](370, 50)(355, 35)(340, 20)
    \qbezier[40](310, 50)(335, 37.5)(358, 26)
    \qbezier[40](370, 50)(346, 62)(322, 74)
    \put(315, 74){\tiny F}
    \put(331, 17){\tiny F$'$}
    \put(340, 20){\circle{1}}
    \put(322, 74){\circle{1}}
    \put(342, 79){\tiny G}
    \put(360, 24){\tiny G$'$}
    \put(358, 26){\circle{1}}
    \put(340, 80){\circle{1}}
    \put(346, 2){\line(-1, 3){30}}
    \put(366, 2){\line(-1, 3){32}}
    \linethickness{0.3mm}
    \qbezier(352, 2)(338, 58)(326, 106)
\end{picture} \] \begin{center} Figure 2\end{center}

\begin{conj} {\bf The map of cycles in the upper half-plane.} (See
Figure 3). The infinite polygonal domain $QNMR$ minus the segment
$[G,K]$ and the ray $FR$ is the $\alpha$-set. The union of open
triangles $\Delta AFM$ and $\Delta BGN$ and the open infinite
sector with vertex $K$, bounded by rays $KP$ and $KQ$, is the
$\beta$-set. All other points of the upper half-plane belong to
the $\gamma$-set. \end{conj}
\[\begin{picture}(400,100) \put(200,5){\vector(0,1){100}}
\put(140,40){\line(3,-1){60}} \put(140,40){\line(3,2){60}}
\put(200,20){\line(3,2){60}} \put(200,80){\line(3,-1){60}}
\put(200,50){\line(6,1){180}} \put(200,50){\line(-6,-1){180}}
\put(170,60){\line(-3,-1){140}} \put(170,60){\line(3,1){100}}
\put(230,40){\line(3,1){140}} \put(230,40){\line(-3,-1){100}}
\put(202,14){\tiny A} \put(194,81){\tiny B} \put(139,33){\tiny C}
\put(258,62){\tiny C${}'$} \qbezier[50](200,20)(185,40)(170,60)
\qbezier[50](200,80)(215,60)(230,40)
\qbezier[60](200,20)(207,48)(214,75)
\qbezier[60](200,80)(193,52)(185,25) \put(167,62){\tiny G}
\put(212,77){\tiny G${}'$} \put(184,18){\tiny F}
\put(228,33){\tiny F${}'$} \put(80,33){\tiny K}  \put(14,18){\tiny
P} \put(23,11){\tiny Q} \put(204,103){\tiny x} \put(194,31){\tiny
M} \put(202,65){\tiny N} \put(123,5){\tiny R}
\end{picture}\]
\begin{center} Figure 3 \end{center}

\begin{rem} Situation in the lower half-plane is much more
complicated. In particular, there are points \emph{inside} $\Delta
ABC'$ that belong to the $\gamma$-set. \end{rem}

\begin{rem} Cases of a right or an obtuse triangles $ABC$ are also much
more complicated.
\end{rem}

\begin{ex} Let us consider the acute triangle $\Delta ABC$,
$A=(0,0)$, $B=(15,0)$, $C=(5,10)$. The line $l: 3x+y=\frac{45}{2}$
is the central line of the $\alpha$-set in the upper half-plane.
Here is the plot of the function $a(y)$, $(x,y)\in l$, $0\leqslant
y<\infty$\,:
\[\begin{picture}(140,150) \put(0,10){\vector(1,0){140}}
\put(134,4){\tiny y} \put(10,2){\vector(0,1){145}}
\put(3,142){\tiny a} \qbezier(10,85)(50,75)(81,51)
\qbezier(81,51)(90,49)(100,47) \qbezier(100,47)(115,80)(130,140)
\qbezier[40](81,10)(81,30))(81,51)
\qbezier[35](100,10)(100,24)(100,47) \put(79,4){\tiny $u$}
\put(99,4){\tiny $v$} \end{picture}\]
\begin{center} Figure 4\end{center}
\pn Here $a(0)=7.5$; $a(u)\approx 4.1$, where $u\approx 7.1$ is
the $y$-coordinate of the intersection point of $l$ and the circle
$x^2+y^2-45x=0$; $a(v)\approx 3.7$,where $v=9$ is the
$y$-coordinate of the intersection point of $l$ and $AC$.
\end{ex}

\section{Theorems}
\pn Here we will prove three results about 4-cycles in special
cases.

\begin{prop} There are no 4-cycles in a right pyramid, i.e. in a
section of the first octant. \end{prop}
\begin{proof} Let $P=ABCO$ be  a pyramid, where $A$ is a point in the
positive $x$-axis, $B$ --- a point in the positive $y$ --- axis,
$C$ --- a point in the positive $z$-axis and $O$ is the origin. A
billiard trajectory begins at a point $S\in \Delta ABC$ with a
vector $\bar v$. After three reflections from planes $xy$, $xz$
and $yz$ it returns to $\Delta ABC$ with the directing vector
$-\bar v$. The reflection from $\Delta ABC$ must transforms $-\bar
v$ into $\bar v$, hence, the vector $\bar v$ is orthogonal to
$ABC$. But in this case the returning point cannot be $S$.
\end{proof}

\begin{prop} If a pyramid has the right dihedral angle, then it
has not more, than two 4-cycles. \end{prop}
\begin{proof} Let $ABCD$ be a pyramid, where $A=(0,0,0)$,
$B=(a,0,b)$, $C=(0,c,d)$ and $D=(0,0,e)$, i.e. the dihedral angle
at the edge $AD$ is $\frac \pi2$. We will consider 4-cycles
$$ABC\to ABD\to ACD\to BCD\to ABD \text{ and } ABC\to ACD\to
ABD\to BCD\to ABC$$ and will prove that both two cannot exist.
\pmn As faces $ABD$ and $ACD$ are orthogonal, then reflection
operators with respect to these faces commute. But then these
cycles have the same starting vector and, thus, the same starting
point. \end{proof} \pn Let a pyramid $ABCD$ be symmetric with
respect to the plane $CDE$, where $E$ is the middle point of the
edge $AB$. Then the altitude $DO$ belongs to $CDE$, lines $CE$ and $AB$ are
perpendicular and lines $DE$ and $AB$ are also perpendicular. Let
us consider the cycle $\mathcal{C}=ABC\to ACD\to ABD\to BCD\to
ABC$. Let $K,L,M,N$ be points of $\mathcal{C}$ that belong to
$\Delta ABC$, $\Delta ACD$, $\Delta ABD$ and $\Delta BCD$,
respectively. The cycle $ABC\to BCD\to ABD\to ACD\to ABC$ is the
same cycle, passed in the reversed direction. As the symmetry with
respect to the plane $CDE$ maps the first cycle into the second,
then the cycle $\mathcal{C}$ is symmetric: $K\in [C,E]$, $M\in
[D,E]$ and points $L$ and $N$ are symmetric. As the plane $LMN$
contains the normal vector to the plane $ABD$, then $LM\bot DE$.
\pmn Let us consider the pyramid $ACDB'$ which is symmetric to
$ABCD$ with respect to the plane $ACD$, and let points $E'$ and
$M'$ be symmetric to $E$ and $M$, respectively. Points $K,L,M'$
are collinear (because $KLM$ is the billiard trajectory) and $LM'\bot
DE'$ (because $LM\bot DE$). Thus, the starting vector
$\overline{KL}$ is orthogonal both to $CE$ and $DE'$.
\[\begin{picture}(155,150) \put(10,40){\line(0,1){75}}
\put(10,40){\line(4,-1){120}} \put(10,40){\line(5,6){75}}
\put(130,10){\line(1,6){15}} \put(85,130){\line(2,-1){60}}
\put(10,115){\line(5,1){75}} \qbezier[90](10,40)(77,70)(145,100)
\qbezier(85,130)(107,70)(130,10)
\qbezier[90](10,115)(77,107)(145,100)
\qbezier[40](70,25)(107,62)(145,100)
\qbezier[40](70,25)(77,77)(85,130)
\qbezier[40](85,130)(47,103)(10,77) \put(8,33){\tiny A}
\put(128,3){\tiny B} \put(148,98){\tiny C} \put(83,133){\tiny D}
\put(68,18){\tiny E} \put(0,113){\tiny B${}'$} \put(0,74){\tiny
E${}'$} \put(100,55){\circle*{2}} \put(100,48){\tiny K}
\put(100,55){\line(-4,3){30}} \put(70,77){\circle*{2}}
\put(70,80){\tiny L} \put(74,54){\circle*{2}}
\qbezier(70,77)(72,63)(74,54) \put(66,52){\tiny M}
\put(125,70){\circle*{2}} \qbezier(74,54)(100,62)(125,70)
\put(128,68){\tiny N} \put(100,55){\line(5,3){25}}
\put(42,98){\circle*{2}} \qbezier[30](70,77)(56,87)(42,98)
\put(33,99){\tiny M${}'$}
\end{picture}\]
\begin{center}Figure 5\end{center}
\pn Now we can describe the geometry of our 4-cycle.
\begin{prop} The staring point and the starting vector belongs to
the common perpendicular of two skew lines $CE$ and $DE'$.
\end{prop}

\begin{conj} If a pyramid has $k$ obtuse dihedral angles, then the
number of its 4-cycles is not more, than $3-k$. \end{conj}

\section{Physical billiard. Computations}
\pn In this section we will study the movement of a mass point in
the presence of a gravity field with the constant $g$. A point
moves inside a triangular pyramid with elastic reflections (i.e.
without loss of the energy), thus, trajectories of the point
between reflections are segments of parabolas. \pmn It must be
noted that each billiard trajectory can be passed in both
directions. Thus, we must consider three different orders of
reflections. As opposed to the mathematical case, each order of
reflections can admit infinitely many trajectories. Computations
in the physical case are much more complex, than in the
mathematical one, because we have to work with \emph{nonlinear}
systems. We present here only results of computations and don't
formulate statements or conjectures. \pmn The above assumptions
about our pyramid are preserved: the base $ABC$ is in the $xy$
plane: $A$ at the origin, $B$ at the positive $x$-axis, $C$ in the
upper half-plane. The apex $D$ is in the upper half-space. The
gravity force with the constant $g$ is directed downwards. \pmn
\begin{ex} Let $A=(0,0,0)$, $B=(4,0,0)$, $C=(3,3,0)$, $D=(2,1,3)$.
Here
$$ABD:3y-z=0,\,\, ACD:3x-3y-z=0,\,\,BCD:-9x-3x-5z+36=0$$ --- are equations of
planes and
$$\bar{n}_1=(0,3,-1),\,\,\bar{n}_2=(3,-3,-1),\,\,\bar{n}_3=(-9,-3,-5)$$
--- are corresponding normal vector (all of them are directed inside
the pyramid). Let the order of reflections be $ABD\to ACD\to
BCD\to ABC$, the starting point be $p_1=(a,b,0)$, the starting
vector be $\bar{v}_1=(k,l,m)$, $m>0$. \pmn We introduce new
variables $t_1$
--- the time interval from the start to the encounter with the
$ABD$ plane, $t_2$ --- the time interval between encounters with
planes $ABD$ and $ACD$, $t_3$ --- the time interval between
encounters with planes $ACD$ and $BCD$ and $t_4$
--- the time interval between the encounter with the plane $BCD$
and the return to the starting point. \pmn The coordinates of the
velocity vector $\bar{v}_2$ in the moment of the encounter with
the $ABD$-plane are $\bar{v}_2=(k,l,m-g\cdot t_1)$. At this moment
the position $p_2$ of our mass point has coordinates
$$p_2=\left(a+k\cdot t_1,b+l\cdot t_1,m\cdot t_1-g\cdot
t_1^2/2\right)\,.$$ Thus, we have the first equation:
$$3\cdot(b+l\cdot t_1)-\left(m\cdot
t_1-g\cdot t_1^2/2\right)=0  \eqno({\rm Eq. 1})$$ Let $s$ be the
dot product $s=(\bar{v}_2,\bar{n}_1)$, then
$$\bar{v}_3=\left(k,l-\frac{6s}{10},m-g\cdot
t_1+\frac{2s}{10}\right)$$ is the velocity vector after the
reflection from the $ABD$-plane. \pmn The next step: we find the
velocity vector $\bar{v}_4$ in the moment $t_1+t_2$, i.e. in the
moment of the contact with the plane $ACD$:
$$\bar{v}_4=(\bar{v}_3[1],\bar{v}_3[2],\bar{v}_3[3]-g\cdot t_2)$$
and the position $p_3$ of our mass point at this moment:
$$p_3=(p_2[1]+\bar{v}_3[1]\cdot t_2,p_2[2]+\bar{v}_3[2]\cdot
t_2,p_2[3]+\bar{v}_3[3]\cdot t_2-g\cdot t_2^2/2).$$ This gives us
the second equation:
$$3\cdot p_3[1]-3\cdot p_3[2]-p_3[3]=0. \eqno({\rm Eq. 2})$$ In
the same manner we obtain the third equation. \pmn The contact of
the mass point with the base $ABC$ gives us 6 equations: we must
come to the starting point with the velocity vector
$\bar{v}_8=(k,l,-m)$. Thus, we have a nonlinear system of nine
equations and ten variables $\{a,b,k,l,m,g,t_1,t_2,t_3,t_4\}$.
\end{ex}
\pn Given such system we find the Groebner basis for the
lexicographic order $\{a,b,k,l,m,g,t_4,t_3,t_2,t_1\}$ and get the
following results (in generic case).
\begin{enumerate}
\item We have two independent variables $t_2$ and $t_3$ and all
other variables are their rational functions. \item Numerators and
denominators of these fractions  are homogeneous polynomials in
$t_2$ and $t_2$:
\begin{enumerate}
\item the degrees of numerators and denominators of $t_1$ and
 $t_4$ are two and one, respectively;
\item the degrees of the numerator and the denominator of $g$ are
 three and  five, respectively;
\item the degrees of numerators and denominators of $k$, $l$ and
 $m$ are four and five respectively;
\item the degrees of numerators and denominators of $a$ and $b$
are five and five.
\end{enumerate}
\item The geometry of a trajectory depends only on the ratio
$t=t_3/t_2$. The parameter $t_2$ defines the duration of the cycle
passage.
\end{enumerate}
\pn {\bf The continuation of Example 5.1.} Demands on the
positivity of $g,t_1,t_2,t_3,t_4$ and demands on the positions of
reflection points (inside faces) define an admissible interval for
$t$. In our case $0<t<0.48$. Admissible starting points constitute
a curve $s(t)$ inside $\Delta ABC$, where $G=s(0)\approx(2,1)$ and
$H=s(0.48)\approx(2.6,0.8)$ (Figure 6).
\[\begin{picture}(100,80) \put(10,10){\line(1,0){80}}
\put(10,10){\line(1,1){60}} \put(90,10){\line(-1,3){20}}
\put(5,3){\tiny A} \put(92,3){\tiny B} \put(69,73){\tiny C}
\put(50,30){\circle*{2}} \put(62,26){\circle*{2}}
\qbezier(50,30)(56,30)(62,26) \put(43,28){\tiny G}
\put(64,24){\tiny H} \end{picture}\]
\begin{center} Figure 6\end{center}

\begin{rem} For the above pyramid there exist three families of
billiard trajectories for each order of reflections.\end{rem}
\begin{rem} The pyramid $ABCD$: $A=(0,0,0)$, $B=(4,0,0)$, $C=(3,3,0)$,
$D=(3,2,1)$ has three obtuse dihedral angles and there are no
4-cycles in it.\end{rem} \begin{rem} The pyramid $ABCD$:
$A=(0,0,0)$, $B=(9,0,0)$, $C=(6,3,0)$, $D=(6,2,4)$ has obtuse
angle $ACB$ in the base $\Delta ABC$ and only one order of
reflections: $\Delta ABC \to \Delta ACD \to \Delta ABD \to \Delta
BCD \to \Delta ABC$ produces a family of 4-cycles. \end{rem} \pn
Outside the scope of generic cases, symmetric cases are the most
interesting.
\begin{ex} Let us consider the symmetric pyramid $ABCD$:
$A=(0,0,0)$, $B=(6,0,0)$, $C=(3,4,0)$, $D=(3,2,4)$. Let $CH$ be
the altitude of $\Delta ABC$. There exists a family of 4-cycles
for the reflection order $\Delta ABC\to \Delta ACD \to \Delta ABD
\to \Delta BCD \to \Delta ABC$. The starting point always belongs
to $CH$, the starting vector is orthogonal to the $y$-axis. A
trajectory meets $\Delta ABD$ at a point with $x$-coordinate 3.
\end{ex}

\vspace{1cm}
\end{document}